\documentclass[11pt]{elsarticle}

\usepackage{amssymb,amsfonts,amsmath}

\usepackage[latin1]{inputenc}
\usepackage{color}

\newtheorem{theorem}{Theorem}
\newtheorem{claim}{Claim}
\newtheorem{proposition}[theorem]{Proposition}

\newtheorem{corollary}[theorem]{Corollary}

\newtheorem{definition}[theorem]{Definition}
\newtheorem{conjecture}{Conjecture}
\newproof{proof}{Proof}

\def\Even{\mathop{\rm Even}\nolimits}

\begin{document}

\begin{frontmatter}
\title{Every triangle-free induced subgraph of the triangular lattice is $(5m,2m)$-choosable}
\author[MRS,TLN]{Yves Aubry}
\ead{yves.aubry@univ-tln.fr}
\author[TLN]{Jean-Christophe Godin}
\ead{godinjeanchri@yahoo.fr}
\author[DJN]{Olivier Togni}
\ead{olivier.togni@u-bourgogne.fr}
\address[MRS]{Institut de Mathématiques de Luminy, CNRS, Marseille, France}
\address[TLN]{Institut de Math\'ematiques de Toulon, Universit\'e du Sud Toulon-Var,  France}
\address[DJN]{Laboratoire LE2I, CNRS, Universit\'e de Bourgogne, France}



\date{\today}

\begin{abstract}
A graph $G$ is $(a,b)$-choosable if for any color list of size $a$ associated with each vertex, one can choose a subset of $b$ 
colors such that adjacent vertices are colored with disjoint color sets.
This paper proves that for any integer $m\ge 1$, every finite triangle-free induced subgraph of the triangular lattice is $(5m,2m)$-choosable.
\end{abstract}

\begin{keyword}
Radio channel assignment, triangular lattice, choosability, weighted graph.
\MSC[2010]{05C15, 05C38.}
\end{keyword}

\end{frontmatter}

\section{Introduction}
Let $G=(V(G),E(G))$ be a graph  where $V(G)$ is the set of vertices and $E(G)$ is the set of edges, and let $a$, $b$, $n$ and $e$ be integers. 

Given a list assignment (called simply a list) $L$ of $G$ {\em i.e.} a map $L : V(G) \rightarrow \mathcal{P}({\mathbb N})$ and a weight function $w$
of $G$ {\em i.e.} a map $w : V(G) \rightarrow {\mathbb N}$, an {\em $(L,w)$-coloring} $c$ of $G$ is a list 
of the weighted graph $G$ such that for all  $v\in V(G)$, 
$$c(v) \subset L(v) \text{ and } | c(v) | =w(v),$$ and for all $vv' \in E(G)$,  
$$c(v) \cap c(v')  = \emptyset.$$
We say that  $G$ is {\em $(L,w)$-colorable} if there exists an $(L,w)$-coloring $c$ of $G$.
An $(L,b)$-coloring $c$ of $G$ is an $(L,w)$-coloring such that for all $v \in V(G)$, we have $w(v)=b$.
A $a$-list $L$ of $G$ is a list of $G$ such that for all $v \in V(G)$, we have  $|L(v) |=a$.
The graph $G$ is said to be {\em $(a,b)$-choosable} if for any $a$-list $L$ of $G$, there exists an $(L,b)$-coloring $c$ of $G$.
If the graph is $(a,b)$-choosable for the $a$-list $L$ such that $L(v)=L(v')$ for all vertices $v,v'$, then $G$ is $(a,b)$-colorable.

The concept of choosability of a graph, also called list coloring, has been introduced by Vizing \cite{viz76}, and independently by Erd\H{o}s, Rubin and Taylor~\cite{erdo79}. It contains of course the colorability as a particular case. Since its introduction, choosability has been extensively studied (see for example \cite{alontuzavoigt97, borodinkostochkawoodall97, thom03, tuz96, gravier96} and more recently \cite{GutnerTarsi, hav09}).
Even for the original (unweighted) version, the problem proves to be difficult, and is NP-complete for very restricted graph classes. 

Every graph that is $(a,b)$-colorable is also trivially $(am,bm)$-colorable for any integer $m\ge 1$. For list coloring, Erd\H{o}s, Rubin and Taylor ~\cite{erdo79} conjectured the following:
\begin{conjecture}[\cite{erdo79}]
The $(a,b)$-choosability of a graph $G$ implies $(am,bm)$-choosability for all $m\in\mathbb{N}, \ m\ge 1$. 
\end{conjecture}
In relation with this question, Gutner and Tarsi~\cite{GutnerTarsi} have recently exhibited graphs $G$ that are $(a,b)$-choosable but not $(c,d)$-choosable, with $\frac cd >  \frac ab \ge 3$.

List multicoloring problems on graphs can be used to model channel assignment problems in wireless systems. Sets of radio frequencies are to be assigned to transmitters such that adjacent transmitters are assigned disjoint sets of frequencies. Often these transmitters are laid out like vertices of a triangular lattice in a plane. This problem corresponds to the problem of multicoloring an induced subgraph of a triangular lattice with integer demands associated with each vertex. 
Since more than a decade, multicoloring of subgraphs of the triangular lattice has been the subject of many papers (see e.g.~\cite{mcd00, hav01, hav02, KT06}). Mc Diarmid and Reed~\cite{mcd00} have made the following conjecture when the subgraph is induced and contains no triangle:
\begin{conjecture}[\cite{mcd00}]
\label{c1}
Every triangle-free induced subgraph of the triangular lattice is $(\lceil\frac{9b}{4}\rceil,b)$-colorable.
\end{conjecture}

Some progress have been made regarding this conjecture with Havet~\cite{hav01} proving the $(5,2)$-colorability and $(7,3)$-colorability and Sudeep and Vishwanathan~\cite{Sudeep} (with a simpler proof) the $(14,6)$-colorability of any triangle-free induced subgraph of the triangular lattice.

The main result is Theorem~\ref{theorem52} of Section~\ref{triangle} which shows the $(5m,2m)$-choosability of
triangle-free induced subgraphs of the triangular lattice. The method is similar with that
of Havet~\cite{hav01}, that uses precoloring extensions and decomposition into induced path (called handles). However,
we need here specific results on the choosability of a weighted path. We find convinient to work on a type of list of
the path that we call {\em waterfall} list. This is the subject of the next section.

\section{Waterfall lists on the path}
\label{path}
We first define the similarity of two lists with respect to a weighted graph:

\begin{definition} Let $(G,w)$ be a weighted graph. Two lists $L$ and $L'$ are said to be {\em similar} if this assertion is true:\\
$$G \text{ is }(L,w)\text{-colorable }\Leftrightarrow G  \text{ is }(L',w)\text{-colorable}.$$
\end{definition}

The \textsl{path} $P_{n+1}$ of length $n$ is the graph with vertex set $V=\{v_{0},v_{1},\dots,v_{n}\}$ and edge set
$E=\bigcup_{i=0}^{n-1} \{v_{i}v_{i+1}\}$. To simplify the notation, $L(i)$ denotes $L(v_{i})$, $c(i)$ denotes
$c(v_{i})$ and $w(i)$ denotes $w(v_i)$.\\ 

By analogy with the flow of water in waterfalls, we define a waterfall list as follows:
\begin{definition}
A {\em waterfall} list $L$ of a path $P_{n+1}$ of length $n$ is a list $L$ such that for all $i,j \in \{0,\dots,n\}$ 
with $|i-j| \geq 2$, we have $L(i) \cap L(j) = \emptyset$.
\end{definition}
Notice that another similar definition of a waterfall list is that any color is present only on one 
list or on two lists of consecutive vertices. Figure 1 shows a list $L$ of the path $P_5$ (on the left), together with a similar waterfall list $L^c$ (on the right).

\begin{definition} For a weighted path $(P_{n+1},w)$,
\begin{itemize}
\item A list $L$ is {\em good} if $|L(i)|\geq w(i) + w (i+1)$ for any $i, 1\leq i\leq n-1$.
\item The {\em amplitude} $A(i,j)(L)$ (or $A(i,j)$) of a list $L$ is $A(i,j)(L)=\cup_{k=i}^{j}L(k)$.
\end{itemize}
\end{definition}

\unitlength=0.8cm
\begin{picture}(20,8)
\put(0.1,5.3){$L( )$}
\put(1.9,7.5){$v_0$}
\put(2.7,7.5){$v_1$}
\put(3.5,7.5){$v_2$}
\put(4.3,7.5){$v_3$}
\put(5.1,7.5){$v_4$}
\put(2,7.3){\circle*{0.2}}
\put(2.8,7.3){\circle*{0.2}}
\put(3.6,7.3){\circle*{0.2}}
\put(4.4,7.3){\circle*{0.2}}
\put(5.2,7.3){\circle*{0.2}}
\put(2,7.3){\line(1,0){3.2}}

\put(1.8,6.5){1}\put(3.4,6.5){1}\put(4.2,6.5){1}\put(5,6.5){1}
\put(1.8,6){2}\put(2.6,6){2}
\put(1.8,5.5){3}\put(2.6,5.5){3}\put(3.4,5.5){3}\put(4.2,5.5){3}
\put(1.8,5){4}\put(2.6,5){4}\put(4.2,5){4}\put(5,5){4}
\put(2.6,4.5){5}\put(3.4,4.5){5}\put(5,4.5){5}
\put(1.8,4){6}\put(3.4,4){6}\put(5,4){6}
\put(3.4,3.5){7}
\put(7.5,4){\vector(1,0){1}}
\put(8.5,4){\vector(-1,0){1}}
\put(7.3,3.5){similar}
\put(10,5.3){$L^{c}( )$}

\put(11.9,7.5){$v_0$}
\put(12.7,7.5){$v_1$}
\put(13.5,7.5){$v_2$}
\put(14.3,7.5){$v_3$}
\put(15.1,7.5){$v_4$}
\put(12,7.3){\circle*{0.2}}
\put(12.8,7.3){\circle*{0.2}}
\put(13.6,7.3){\circle*{0.2}}
\put(14.4,7.3){\circle*{0.2}}
\put(15.2,7.3){\circle*{0.2}}
\put(12,7.3){\line(1,0){3.2}}

\put(11.8,6.5){1}
\put(11.8,6){2}
\put(11.8,5.5){3}\put(12.6,5.5){3}
\put(11.8,5){4}\put(12.6,5){4}
\put(11.8,4.5){5}\put(12.6,4.5){5}
\put(12.6,4){6}\put(13.4,4){6}
\put(13.4,3.5){7}
\put(13.4,3){8}
\put(13.4,2.5){9}\put(14.2,2.5){9}
\put(14.2,2){10}\put(15,2){10}
\put(14.2,1.5){11}\put(15,1.5){11}
\put(15,1){12}
\put(15,0.5){13}

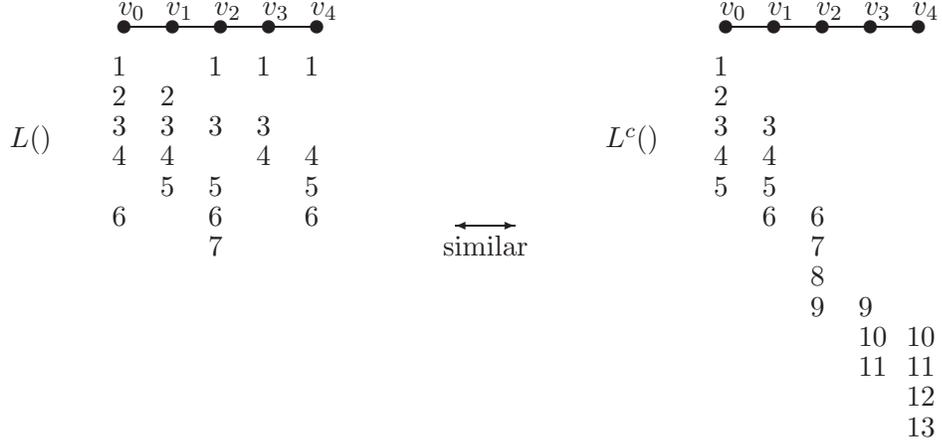
\put(1,-0.5){Fig. 1. Example of a list $L$ which is similar to a waterfall list $L^c$.}
\end{picture}
\bigskip


We first show that any good list can be transformed into a similar waterfall list.
\begin{proposition}\label{similar}
\label{llf}
 For any good list $L$ of $P_{n+1}$, there exists a similar waterfall list $L^c$ with $\vert L^c(i)\vert=\vert L(i)\vert$ for all $i\in\{0,\ldots,n\}$.
\end{proposition}

\begin{proof}
We are going to transform a good list $L$ of $P_{n+1}$  into a waterfall list $L^c$ and we will prove that $L^c$ is similar with $L$.

First, remark that if a color $x\in L(i-1)$ but $x\not\in L(i)$ for some $i$ with  $1\leq i \leq n-1$, then for any $j>i$, one can change the color $x$ by a new color 
$y\not\in A(0,n)(L)$ in the list $L(j)$, without changing the choosability of the list.
With this remark in hand, we can assume that $L$ is such that any color $x$ appears on the lists of consecutive vertices $i_x,\ldots , j_x$.

Now, by permuting the colors if necessary, we can assume that if $x<y$ then $i_x < i_y$ or $i_x =i_y$ and $j_x\leq j_y$.

Repeat the following transformation:

1. Take the minimum color $x$ for which $j_x \geq i_x +2$ i.e. the color $x$ is present on at least three vertices 
$i_x, i_x +1 ,i_x +2,\ldots , j_x$;

2. Replace color $x$ by a new color $y$ in lists $L(i_x +2),\ldots , L(j_x )$;

\noindent
until the obtained list is a waterfall list (obviously, the number of iterations is always finite).

Now, we show that this transformation preserves the choosability of the list.
Let $L'$ be the list obtained from the list $L$ by the above transformation.

If $c$ is an $(L,w )$-coloring of $P_{n+1}$  then the coloring $c'$ obtained from $c$ by changing the color $x$ by the color $y$ in the color set 
$c(k)$ of each vertex $k\geq i_x +2$ (containing $x$) is an  $(L',w )$-coloring since $y$ is a new color.

Conversely, if $c'$ is an $(L',w )$-coloring of $P_{n+1}$, we consider two cases:

\medskip
\noindent
{\sl Case 1}: $x\not\in c'(i_x +1)$ or $y\not\in c'(i_x +2)$. In this case, the coloring $c$ obtained from $c'$ by changing the color $y$ by the 
color $x$ in the color set $c'(k)$ of each vertex $k\geq i_x +2$ (containing $y$) is an  $(L,w )$-coloring.

\medskip
\noindent
{\sl Case 2}: $x\in c'(i_x +1)$ and $y\in c'(i_x +2)$. We have to consider two subcases:
\begin{itemize}
 \item Subcase 1: $L'(i_x +1 )\not\subset (c'(i_x ) \cup c'(i_x +1) \cup c'(i_x +2))$. 
There exists $z \in L'(i_x +1 )\setminus (c'(i_x ) \cup c'(i_x +1) \cup c'(i_x +2))$ and the coloring $c$ obtained from $c'$ by changing the color 
$x$ by the color $z$ in $c'(i_x +1)$ and replacing the color $y$ by the color $x$ in the color set $c'(k)$ of each vertex $k\geq i_x +2$ (containing $y$) is an  $(L,w )$-coloring.
 \item Subcase 2: $L'(i_x +1 )\subset (c'(i_x ) \cup c'(i_x +1) \cup c'(i_x +2))$. We have
$$|L'(i_x +1)| = \Big| \Big( (c'(i_x ) \cup c'(i_x +1) \cup c'(i_x +2)\Big) \cap L'(i_x +1)\Big|.  $$ 
As $c'$ is an $(L',w)$-coloring of $P_{n+1}$, we have
$$ |L'(i_x +1)|=|c'(i_x +2)  \cap L'(i_x +1)| + |c'(i_x +1)  \cap L'(i_x +1)|+\Big|\Big( c'(i_x) \backslash c'(i_x +2) \Big) \cap L'(i_x +1)\Big|,$$
$$ |L'(i_x +1)| - w (i_x +1) - |c'(i_x+2)  \cap L'(i_x +1)| = \Big|\Big( c'(i_x) \backslash c'(i_x +2) \Big) \cap L'(i_x +1)\Big|.  $$

Since $y \in c'(i_x +2)$ and $y \notin L'(i_x +1)$, we obtain that 
$$|c'(i_x +2)  \cap L'(i_x +1)| \leq w (i_x +2) - 1,$$ hence
$$ \Big( |L'(i_x +1)| - w (i_x +1) - w (i_x +2) \Big) + 1 \leq \Big| \Big( c'(i_x) \backslash c'(i_x +2) \Big) \cap L'(i_x +1)\Big|.$$
But, by hypothesis, $L$ is a good list. Thus $|L(i_x +1)|=|L'(i_x +1)| \geq w(i_x +1) + w(i_x +2)$ and 
$$ 1 \leq \Big|\Big( c'(i_x) \backslash c'(i_x +2) \Big) \cap L'(i_x +1)\Big| .$$
Consequently, there exists $z \in \Big( c'(i_x) \backslash c'(i_x +2) \Big) \cap L'(i_x +1)$. The coloring $c$ is then constructed from $c'$ by changing
the color $x$ by the color $z$ in $c'(i_x +1)$, the color $z$ by the color $x$ in $c'(i_x)$ and the color $y$ by the color $x$ in the set $c'(k)$ of each vertex $k\geq i_x +2$.
\end{itemize}
\qed
\end{proof}

Cropper et al.~\cite{CGHHJ} consider Philip Hall's theorem on systems of distinct representatives and its improvement by Halmos and Vaughan as statements about the existence of proper list colorings or list multicolorings of complete graphs. The necessary and sufficient condition in these theorems is generalized in the new setting as "Hall's condition'' : 
$$\forall H\subset G, \sum_{k\in C} \alpha(H,L,k) \ge \sum_{v\in V(H)} w(v),$$ where $C=\bigcup_{v\in V(H)}L(v)$ and $\alpha(H,L,k)$ is the independence number\footnote{the independence number of a graph is the size of the largest set of isolated vertices} of the subgraph of $H$ induced by the vertices containing $k$ in their color list. Notice that $H$ can be restricted to be a connected induced subgraph of $G$.

It is easily seen that Hall's condition is necessary for a graph to be $(L,w)$-colorable. Cropper et al.~\cite{CGHHJ} showed that the condition is also sufficient for some graphs, including paths:

\begin{theorem}[\cite{CGHHJ}]
\label{Cropper} For the following graphs, Hall's condition is sufficient to ensure an $(L,w)$-coloring:
\begin{itemize}
 \item[(a)] cliques;
 \item[(b)] two cliques joined by a cut-vertex;
 \item[(c)] paths;
 \item[(d)] a triangle with a path of length two added at one of its vertices;
 \item[(e)] a triangle with an edge added at two of its three vertices.
\end{itemize}
\end{theorem}

This result is very nice, however, it is often hard to compute the left part of Hall's condition, even for paths. However, as the next result shows, Hall's condition is very easy to check when restricted to waterfall lists.


\begin{theorem}
\label{theolistecascadechoisissable}
Let $L^{c}$ be a waterfall list of a weighted path $(P_{n+1},w)$. 
Then $P_{n+1}$ is $(L^{c},w)$-colorable if and only if:
$$\forall i,j \in \{0,\dots,n\},  \: |\bigcup_{k=i}^{j}L^{c}(k)| \geq \sum_{k=i}^{j} w (k).$$ 
\end{theorem}

\begin{proof}
``if'' part: Recall that $A(i,j)=\cup_{k=i}^{j}L^c(k)$.
For $i,j \in \{0,\dots,n\}$, let $P_{i,j}$ be the subpath of $P_{n+1}$ induced by the vertices $i,\ldots, j$.
By Theorem \ref{Cropper}, it is sufficient to show that 
$$\forall i,j \in \{0,\dots,n\},  \: \sum_{x\in A(i,j)} \alpha(P_{i,j},L^c,x) \geq \sum_{k=i}^{j} w (k).$$
Since the list is a waterfall list, then for each color $x\in A(i,j)$, $\alpha(P_{i,j},L^c,x)=1$ and thus 
$\sum_{x\in A(i,j)} \alpha(P_{i,j},L^c,x) = |A(i,j)| = |\bigcup_{k=i}^{j}L^{c}(k)|$.

``only if'' part: If $c$ is an $(L^{c},w)$-coloring of $P_{n+1}$  then
$$\forall i,j \in \{0,\dots,n\} : \: \bigcup_{k=i}^{j}L^{c}(k) \supset \bigcup_{k=i}^{j} c(k).$$
Since $L^c$ is a waterfall list, it is easily seen that $|\bigcup_{k=i}^{j} c(k)|=\sum_{k=i}^{j} w(k)$. Therefore, 
$\forall i,j \in \{0,\dots,n\} : \: |\bigcup_{k=i}^{j}L^{c}(k)| \geq \sum_{k=i}^{j} w (k)$.
\qed\end{proof}

Theorem~\ref{theolistecascadechoisissable} has the following corollary when the list is a good waterfall list and 
 $|L(n)| \geq w(n)$. 

\begin{corollary}
\label{lemmeencascadeequige}
Let $L^{c}$ be a waterfall list of a weighted path $(P_{n+1},w)$ such that for any 
$i, 1\leq i\leq n-1$, $|L^{c}(i)|\geq w(i) + w (i+1)$ and
$|L^c(n)| \geq w(n)$. Then $P_{n+1}$ is  $(L^{c},w)$-colorable if and only if  
$$ \forall j \in \{0,\dots,n\}, \:|\bigcup_{k=0}^{j}L^{c}(k)| \geq \sum_{k=0}^{j} w (k). $$ 
\end{corollary}

\begin{proof}
Under the hypothesis, if $P_{n+1}$ is $(L^{c},w)$-colorable, then Theorem 
 \ref{theolistecascadechoisissable} proves in particular the result.\\
 
 Conversely, since
 $L^{c}$ is a waterfall list of  $P_{n+1}$, we have:
$$\forall i,j \in \{1,\dots,n\}, \: |A(i,j)|=|\cup_{k=i}^{j} L^c(k)|  \geq |\cup_{\substack{k=i \\ k-i \ even}}^{j} L^c(k)| =\sum_{\substack{k=i \\ k-i  \ even}}^{j} |L^{c}(k)|. $$
Since $L^c$ is a good list of   $P_{n+1}$  (for simplicity, we set   $w(n+1)=0$):
$$ \forall i,j \in \{1,\dots,n\}, \: \sum_{\substack{k=i \\ k-i  \ even}}^{j} |L^{c}(k)| \geq \sum_{\substack{k=i \\ k-i  \ even}}^{j} (w (k) + w (k+1))  \geq \sum_{k=i}^{j} w (k),$$
and then we obtain for all  $ i,j \in \{1,\dots,n\}, \: |A(i,j)| \geq \sum_{k=i}^{j} w (k)$. Since for all $ j \in
\{0,\dots,n\}, \:|A(0,j)| \geq \sum_{k=0}^{j} w (k) $, Theorem \ref{theolistecascadechoisissable} concludes the proof.
\qed\end{proof}

\bigskip

Another interesting corollary holds for lists $L$ such that 
 $|L(0)|=|L(n)|= b$, and for all $i \in \{1,\dots,n-1\},  |L(i)|=a$. The function $\Even$ is defined for any real $x$ by: $\Even(x)$ is the smallest even integer $p$ such that $p\geq x$.

\begin{corollary}
\label{theorem48these}
Let $L$ be a list of $P_{n+1}$ such that $|L(0)|=|L(n)|= b$,
and $|L(i)|=a=2b+e$ for all $i \in \{1,\dots,n-1\}$ (with $e\not=0$). 
\begin{center}
If $n \geq \Even\Bigl(\frac{2b}{e}\Bigr)$ then $P_{n+1}$ is $(L,b)$-colorable.                                                                              \end{center}
\end{corollary}

\begin{proof}
The hypothesis implies that $L$ is a good list of $P_{n+1}$, hence by Proposition \ref{llf}, there exists a waterfall list $L^c$ similar to $L$. So we get:
$$\forall i \in \{1,\dots,n-1\}, \ |L^{c}(i)| \geq 2b = w (i) + w (i+1)$$
and $|L^{c}(n)| \geq b=w (n)$. By Corollary  \ref{lemmeencascadeequige} it remains to prove that: 
$$\forall j \in \{0,\dots,n\},  \:|A(0,j)| \geq \sum_{k=0}^{j} w (k) = (j+1)b.$$ 

{\sl Case 1}: $j=0$. By hypothesis, we have  $|A(0,0)|=|L^{c}(0)| \geq b$.\\

{\sl Case 2}: $j \in \{1,\dots,n-1\}$. Since $L^{c}$ is a waterfall list of $P_{n+1}$ we obtain that: \\
if $j$ is even
$$|A(0,j)| \geq \sum_{\substack {k=0 \\ k\ even}}^j |L^c(k)| = b + \sum_{\substack{k=2 \\ k  \ even}}^{j} 2b = b + \frac{j}{2} 2b = (j+1)b,$$
and if $j$ is odd
$$ |A(0,j)| \geq  \sum_{\substack {k=0 \\ k\ odd}}^j |L^c(k)| = \sum_{\substack{k=1 \\ k  \ odd}}^{j} 2b = \frac{j+1}{2} 2b =(j+1)b.  $$
Hence for all $j \in \{0,\dots,n-1\}, \ |A(0,j)| \geq  (j+1)b$.

{\sl Case 3}: $j=n$. Since $n \geq \Even\Bigl(\frac{2b}{e}\Bigr)$ by hypothesis, and 
$$\mid A(0,n) \mid \geq \sum_{\substack {k=0 \\ k\ odd}}^n |L^c(k)|=\left \{
\begin{array}{ll}
a\frac{n}{2}, & \text{ if $n$ is even}; \\
b+a\frac{n-1}{2}, & \text{ otherwise};
\end{array}
\right.$$
we deduce that $|A(0,n)| \geq (n+1)b$, which concludes the proof.
\qed\end{proof}

%

Also, we will need the following result with more restrictions on the lists of colors on the two last vertices of the path.
\begin{theorem}
\label{theorem49these}
Let $n,a,b,e$ be four integers such that  $a= 2b+e$. Let $L$ be a list of a path  $P_{n+1}$ such that $|L(0)| = b$, and
for any $i \in \{1,\dots,n-2\} : |L(i)|=a$, $|L(n-1)| =|L(n)|= b + e$, and $|A(n-1,n)(L)| \geq 2b$. 
\begin{center}
 If $n = \Even(\frac{2b}{e})$ then $P_{n+1}$ is $(L,b)$-colorable.
\end{center}
\end{theorem}

\begin{proof}
Since
$$|A(n-1,n)|=|L(n-1)|+|L(n) \backslash L(n-1)| = (b+e)+|L(n) \backslash L(n-1)| \geq 2b,$$
we obtain that $|L(n) \backslash L(n-1)| \geq b-e$. Let $D$ be a set such that 
$D \subset L(n) \backslash L(n-1) \ and \ |D|=b-e$.

Let $L'$  be the new list constructed with  $L$ such that $L'(i) =L(i)$ if $i \in \{0,\dots,n-1\}$ and $L'(n)= L(n)
\backslash D$ and let $w'$ be a new weight function defined by $w'(i) =b$ if $i \in \{0,\dots,n-1\}$and $w'(n)=e$.

We are going to prove that if $P_{n+1}$ is $(L',w')$-colorable then $P_{n+1}$ is $(L,b)$-colorable.
Indeed, if $c'$ is an $(L',w')$-coloring of $ P_{n+1}$, then we construct $c$ such that:
$$c(i) = \left \{
\begin{array}{ll}
c'(i), & \text{if } i \in \{0,\dots,n-1\}; \\
c'(n) \cup D, & \text{otherwise}.
\end{array}
\right.
$$
Since $D \cap L(n-1) = \emptyset$, we have $c(n-1) \cap c(n) = \emptyset$ and then $c$ is an $(L,b)$-coloring of $
P_{n+1}$.\\
Now, this new list   $L'$ is a good list  of $P_{n+1}$ and $|L'(n)| \geq w'(n)$. 
Proposition \ref{llf} shows that there exists a waterfall list $L^c$ similar to $L'$ such that for all $k$ we have 
 $|L^c(k)|=|L'(k)|$. 
 
Thanks to Corollary \ref{lemmeencascadeequige}, it remains to check that:
$$ \forall j \in \{0,\dots,n\} : \:|A(0,j)(L^{c})| \geq \sum_{k=0}^{j} w' (k). $$

{\sl Case 1}: $j \in \{0,\dots,n-2\}$. Since the list is a waterfall list, we have: 
$$ |A(0,j)(L^{c})| \geq \left \{
\begin{array}{ll}
\sum_{\substack {k=0 \\ k\ even}}^j |L^c(k)|= b+a\frac{j}{2}, &  \text{if $j$ is even}; \\
 \sum_{\substack {k=0 \\ k\ odd}}^j |L^c(k)| = a\frac{j+1}{2}, &  \text{otherwise};
\end{array}
\right.
$$
and the weight function satisfies $\sum_{k=0}^{j} w' (k)=(j+1)b$. Hence, we deduce that 
$$|A(0,j)(L^{c})| \geq \sum_{k=0}^{j} w' (k).$$

{\sl Case 2}: $j=n-1$. Since the list is a waterfall list, we have: 
$$ |A(0,n-1)(L^{c})| \geq \sum_{\substack {k=0 \\ k\ odd}}^{n-1} |L^c(k)|=(b+e)+a\frac{n-2}{2},$$ 
and $\sum_{k=0}^{n-1} w' (k)=nb$. Then $b+e+a\frac{n-2}{2} \geq nb$ if and only if $\frac{en}{2} \geq b$, which is
true by hypothesis since  $n = \Even(\frac{2b}{e})$, thus
$$|A(0,n-1)(L^{c})| \geq \sum_{k=0}^{n-1} w' (k).$$

{\sl Case 3}: $j=n$. Since the list is a waterfall list, we have: 
$$ |A(0,n)(L^{c})| \geq  \sum_{\substack {k=0 \\ k\ even}}^n |L^c(k)|=b+2e+a\frac{n-2}{2},$$
and $\sum_{k=0}^{n} w' (k)=nb+e$. Then $b+2e+a\frac{n-2}{2} \geq nb+e$ if and only if $\frac{en}{2} \geq b$, which
is true by hypothesis since $n = \Even(\frac{2b}{e})$, thus
$$|A(0,n)(L^{c})| \geq \sum_{k=0}^{n} w' (k).$$
\qed\end{proof}

\section{Choosability of the triangular lattice}
\label{triangle}
Let $\mathcal{R}$ be a finite triangle-free induced subgraph of the triangular lattice.
Recall that the triangular lattice is embedded in an Euclidian space and that any vertex $(x,y)$ of $\mathcal{R}$ has at
most six neighbors: 
{its neighbor on the left} $(x-1,y)$, {its neighbor on the right} $(x+1,y)$, {its neighbor on the top left}
$(x-1,y+1)$, {its neighbor on the top right} $(x,y+1)$, {its neighbor on the bottom left}
$(x,y-1)$ and {its neighbor on the bottom right} $(x+1,y-1)$.\\

Follwing the terminology used in~\cite{hav01}, we define nodes and handles of $\mathcal{R}$.
\begin{definition}
The nodes of $\mathcal{R}$ are the vertices of degree  $3$. There are two kinds of nodes: the left nodes whose neighbors
are the neighbors on the left, on the top right, and on the bottom right; and the right nodes whose
neighbors are the neighbors on the right, on the top left and on the bottom left.
\end{definition}

\begin{center}
\unitlength=0.6cm
\begin{picture}(20,4)(0,0)
\put(6.5,2.5){\circle*{0.3}}
\put(7.5,4){\circle*{0.3}}
\put(7.5,1){\circle*{0.3}}
\put(5,2.5){\circle*{0.3}}
\put(13.5,2.5){\circle*{0.3}}
\put(15,2.5){\circle*{0.3}}
\put(12.5,1){\circle*{0.3}}
\put(12.5,4){\circle*{0.3}}
\put(5,2.5){\line(1,0){1.5}}
\put(6.5,2.5){\line(2,3){1}}
\put(6.5,2.5){\line(2,-3){1}}
\put(13.5,2.5){\line(1,0){1.5}}
\put(13.5,2.5){\line(-2,3){1}}
\put(13.5,2.5){\line(-2,-3){1}}
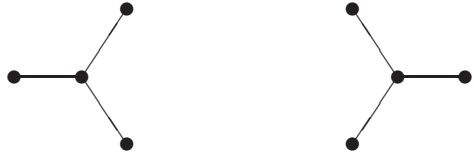
\put(4.8,0){Fig. 2. Left node and right node. }
\end{picture}
\end{center}

\begin{definition}
A {\em cutting node} of  $\mathcal{R}$ is a left node   $(x,y)$ such that for any node   $(x',y')$, we have $y \geq y'$, and
for any left node  $(x',y)$, we have $x' \leq x$. A {\em handle} $H$ of $\mathcal{R}$ is a subpath of $\mathcal{R}$ such that its extremal
vertices are nodes and its internal vertices have degree 2. The set of the internal vertices of a handle $H$ is denoted by $\dot{H}$.

A cutting handle of $\mathcal{R}$  is a handle such that one of its extremal vertices is the cutting node $(x,y)$ and
one of its internal vertices is $(x,y+1)$. 
\end{definition}

One can view the cutting node as the left node the most on the top on the right.

We have the following claim:

\begin{claim}
\label{lemme22these}
Let $H$ be a cutting handle of  $\mathcal{R}$ such that $V(H)=\{v_0,\dots,v_n\}$. If the length $n$ of
$H$ is less than or equal to  $3$, then $n=3$ and $v_{3}$ has a neighbor  $v_4 \not= v_{2} $ of degree less than or
equal to  $2$.
\end{claim}

\begin{center}
\unitlength=0.7cm
\begin{picture}(18,5)(0,0)
\put(6.5,2.5){\circle*{0.3}}
\put(7.5,4){\circle*{0.3}}
\put(9,4){\circle*{0.3}}
\put(10,2.5){\circle*{0.3}}
\put(11.5,2.5){\circle*{0.3}}
\put(9,1){\circle*{0.3}}
\put(7.5,1){\circle*{0.3}}
\put(5,2.5){\circle*{0.3}}
\put(6.5,2.5){\line(2,3){1}}
\put(6.5,2.5){\line(2,-3){1}}
\put(7.5,4){\line(1,0){1.5}}
\put(7.5,1){\line(1,0){1.5}}
\put(10,2.5){\line(1,0){1.5}}
\put(9,1){\line(2,3){1}}
\put(10,2.5){\line(-2,3){1}}
\put(5,2.5){\line(1,0){1.5}}
\put(6.2,2.9){$v_{0}$}
\put(7.2,4.4){$v_{1}$}
\put(8.8,4.4){$v_{2}$}
\put(9.9,2.9){$v_{3}$}
\put(11.4,2.9){$v_{4}$}
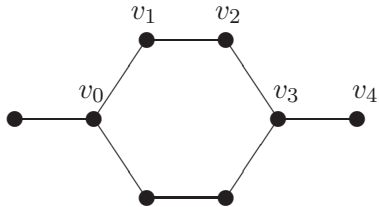
\put(4.2,0){Fig. 3. Cutting handle of length $3$.}
\end{picture}
\end{center}

\begin{theorem}
\label{theorem52}
For any integer $m\ge 1$, every finite triangle-free induced subgraph of the triangular lattice is $(5m,2m)$-choosable.
\end{theorem}

\begin{proof}
Set $a=5m$, $b=2m$ and $e=m$.
Any graph with only vertices of degree 0 or 1 is $(2m,m)$-choosable. Also, a graph having only vertices of degree 2 is a union of cycles. Cycles of even length are $(2m,m)$-choosable and cycles of odd length $\ell>3$ are $(5m,2m)$-choosable~\cite{voi98}.

Let $G$ be a minimal (with respect to the number of vertices) counter-example. 
Since $G$ is a triangle-free induced subgraph of the triangular lattice, its girth\footnote{the length of a shortest (simple) cycle in the graph} is at least 6 and hence it has at least two nodes.
Therefore, by symmetry, $G$ has a cutting handle $H$ (otherwise, consider its mirror graph) and for any $5m$-list
$L$ of $G$, there exists an $(L,2m)$-coloring $c$ of $G-\dot{H}$.
If $H$ is of length $n\ge 4=\Even\Bigl(\frac{4m}{m}\Bigr)$, by Corollary \ref{theorem48these}, $c$ can be extended to an $(L,2m)$-coloring of $G$, contradicting the hypothesis.
Otherwise, by Claim~\ref{lemme22these}, $H$ has length $3$ and $v_3$ has a neighbor $v_4\ne v_2$ of degree two (let $v_5$ be the other neighbor of $v_4$). Hence we have $|L(v_4) \setminus c(v_5)|\ge 5m-2m=3m=b+e$, $|L(v_3)|=5m\ge b+e$ and $|L(v_3) \cup L(v_4)\setminus c(v_5)| \ge |c(v_3 ) \cup c(v_4)| = 4m=2b$.
Then, $H$ satisfies the conditions of Theorem~\ref{theorem49these} and the coloring $c$ restricted to $G-(\dot{H} \cup \{v_3,v_4\})$ can be extended to an $(L,2m)$-coloring of $G$, contradicting the hypothesis.
\qed\end{proof}

Notice that the above proof can be easily translated to show the more general result that for any integers $a,b$ such that $a/b\ge 5/2$, every finite triangle-free induced subgraph of the triangular lattice is $(a,b)$-choosable.

Also, using similar arguments with some additional results on the list-colorability of the path~\cite{god09}, allows us to show that triangle-free induced subgraphs of the triangular lattice are $(7,3)$-colorable, thus giving another proof of Havet's result.

Moreover, the method defined in this paper can serve as a starting tool in order to prove Conjecture~\ref{c1}. Proceeding as in the proof of Theorem~\ref{theorem52} but with $a=9$ and $b=4$ (and $e=a-2b = 1$), allows us to prove that a minimal counter-example to the conjecture does not contain handles of length $n\ge 8$. By using more complex structures than handles, some additional properties of a minimal counter-example were found by Godin~\cite{god09}. However, many configurations remain to be investigated in order to prove the conjecture.



\begin{thebibliography}{3}

\bibitem{alontuzavoigt97} N. Alon, Zs. Tuza, M. Voigt,
\newblock {\sl Choosability and fractional chromatic number},
\newblock Discrete Math. 165/166, (1997), 31-38.


\bibitem{borodinkostochkawoodall97} O.V. Borodin, A.V. Kostochka, D.R. Woodall,
\newblock {\sl List edge and list colourings of multigraph},
\newblock J. Combin. Theory Series B, 71 : 184-204, (1997).

\bibitem{CGHHJ}
M. M. Cropper, J. L. Goldwasser, A. J. W. Hilton, D. G. Hoffman, P. D. Johnson,
\newblock {\sl Extending the disjoint-representatives theorems of Hall, Halmos, and Vaughan to list-multicolorings of graphs.}
\newblock J. Graph Theory 33 (2000), no. 4, 199--219.

\bibitem{erdo79} P. Erd\H{o}s, A.L Rubin and H. Taylor,
\newblock {\sl Choosability in graphs},
\newblock Proc. West-Coast Conf. on Combinatorics, Graph Theory and Computing, Congressus Numerantium XXVI, (1979), 125-157.


\bibitem{god09} J.-C. Godin,
\newblock {\sl Coloration et choisissabilit\'e des graphes et applications},
\newblock PhD thesis (in french), Universit\'e du Sud Toulon-Var, France (2009).

\bibitem{gravier96} S. Gravier,
\newblock {\sl A Haj\'os-like theorem for list coloring},
\newblock Discrete Math. 152, (1996), 299-302.


\bibitem{GutnerTarsi} S. Gutner and M. Tarsi,
\newblock {\sl Some results on (a:b)-choosability},
\newblock Discrete Math. 309, (2009), 2260-2270.


\bibitem{hav09} F. Havet,
\newblock {\sl Choosability of the square of planar subcubic graphs with large girth}.
\newblock Discrete Math. 309, (2009), 3553-3563.

\bibitem{hav01} F. Havet,
\newblock {\sl Channel assignement and multicolouring of the induced subgraphs of the triangular lattice}.
\newblock Discrete Math. 233, (2001), 219-233.


\bibitem{hav02} F. Havet, J. Zerovnik.
\newblock {\sl Finding a five bicolouring of a triangle-free subgraph of the triangular lattice}.
\newblock Discrete Math. 244, (2002), 103-108.



\bibitem{KT06} M. Kchikech and O. Togni.
\newblock {\sl Approximation algorithms for multicoloring powers of square and triangular meshes},
\newblock  Discrete Math. and Theoretical Computer Science, Vol. 8 (1):159-172, 2006.


\bibitem{mcd00} C. McDiarmid and B. Reed.
\newblock {\sl Channel assignement and weighted coloring.}
\newblock Networks, 36, (2000), 114-117.


\bibitem{Sudeep} K.S. Sudeep and S. Vishwanathan.
\newblock {\sl A technique for multicoloring triangle-free hexagonal graphs}, 
\newblock Discrete Math. 300, (2005), 256-259.

\bibitem{thom03} C. Thomassen,
\newblock {\sl The chromatic number of a graph of girth 5 on a fixed surface},
\newblock J. Combin. Theory, (2003), 38-71.

\bibitem{tuz96} Zs. Tuza and M. Voigt,
\newblock {\sl Every 2-choosable graph is (2m,m)-choosable},
\newblock J.  Graph Theory 22, (1996), 245-252.

\bibitem{viz76} V. G Vizing,
\newblock {\sl Coloring the vertices of a graph in prescribed colors (in Russian)},
\newblock Diskret. Analiz. No. 29, Metody Diskret. Anal. v Teorii Kodov i Shem 101 (1976), 3-10.

\bibitem{voi96} M. Voigt,
\newblock {\sl Choosability of planar graphs},
\newblock Discrete Math. 150, (1996), 457-460.  

\bibitem{voi98} M. Voigt,
\newblock {\sl On list Colourings and Choosability of Graphs},
\newblock Abilitationsschrift, TU Ilmenau (1998). 



\end{thebibliography}
\end{document}